\documentclass{amsart}
\usepackage{amssymb}
\usepackage{amsfonts}

\newtheorem{theorem}{Theorem}
\theoremstyle{plain}

\newtheorem{definition}{Definition}

\newtheorem{lemma}{Lemma}

\numberwithin{equation}{section}
\input{tcilatex}

\begin{document}
\title[On Co-ordinated Quasi-Convex Functions]{On Co-ordinated Quasi-Convex
Functions}
\author{$^{\blacktriangledown }$M. Emin \"{O}ZDEM\.{I}R}
\address{$^{\blacktriangledown }$Atat\"{u}rk University, K.K. Education
Faculty, Department of Mathematics, 25240, Kampus, Erzurum, Turkey}
\email{emos@atauni.edu.tr}
\author{$^{\spadesuit ,\clubsuit }$Ahmet Ocak AKDEM\.{I}R}
\curraddr{$^{\spadesuit }$A\u{g}r\i\ \.{I}brahim \c{C}e\c{c}en University,
Faculty of Science and Arts, Department of Mathematics, 04100, A\u{g}r\i ,
Turkey}
\email{ahmetakdemir@agri.edu.tr}
\author{$^{\bigstar }$\c{C}etin YILDIZ}
\address{$^{\bigstar }$Atat\"{u}rk University, K.K. Education Faculty,
Department of Mathematics, 25240, Kampus, Erzurum, Turkey}
\email{yildizcetiin@yahoo.com}
\date{December 10, 2010}

\begin{abstract}
In this paper, we give some definitions on quasi-convex functions and we
prove inequalities contain J-quasi-convex and W-quasi-convex functions. We
give also some inclusions.
\end{abstract}

\maketitle

\section{INTRODUCTION}

Let $f:I\subset 
%TCIMACRO{\U{211d} }%
%BeginExpansion
\mathbb{R}
%EndExpansion
\rightarrow 
%TCIMACRO{\U{211d} }%
%BeginExpansion
\mathbb{R}
%EndExpansion
$ be a convex function on the interval of $I$ of real numbers and $a,b\in I$
with $a<b.$ The following double inequality%
\begin{equation}
f\left( \frac{a+b}{2}\right) \leq \frac{1}{b-a}\dint\nolimits_{a}^{b}f(x)dx%
\leq \frac{f(a)+f(b)}{2}  \label{1.1}
\end{equation}%
is well-known in the literature as Hadamard's inequality. We recall some
definitions;

In \cite{PEC}, Pecaric et al. defined quasi-convex functions as following

\begin{definition}
A function $f:\left[ a,b\right] \rightarrow 
%TCIMACRO{\U{211d} }%
%BeginExpansion
\mathbb{R}
%EndExpansion
$ is said quasi-convex on $\left[ a,b\right] $ if%
\begin{equation*}
f\left( \lambda x+(1-\lambda )y\right) \leq \max \left\{ f(x),f(y)\right\} ,%
\text{ \ \ \ \ }\left( QC\right)
\end{equation*}%
holds for all $x,y\in \left[ a,b\right] $ and $\lambda \in \lbrack 0,1].$
\end{definition}

Clearly, any convex function is quasi-convex function. Furthermore, there
exist quasi-convex functions which are not convex.

\begin{definition}
(See \cite{SS1}, \cite{WR}) We say that $f:I\rightarrow 
%TCIMACRO{\U{211d} }%
%BeginExpansion
\mathbb{R}
%EndExpansion
$ is a Wright-convex function or that $f$ belongs to the class $W(I),$ if
for all $x,$ $y+\delta \in I$ with $x<y$ and $\delta >0,$ we have%
\begin{equation*}
f(x+\delta )+f(y)\leq f(y+\delta )+f(x)
\end{equation*}
\end{definition}

\begin{definition}
(See \cite{SS1}) For $I\subseteq 
%TCIMACRO{\U{211d} }%
%BeginExpansion
\mathbb{R}
%EndExpansion
,$ the mapping $f:I\rightarrow 
%TCIMACRO{\U{211d} }%
%BeginExpansion
\mathbb{R}
%EndExpansion
$ is wright-quasi-convex function if, for all $x,y\in I$ and $t\in \left[ 0,1%
\right] ,$ one has the inequality%
\begin{equation*}
\frac{1}{2}\left[ f\left( tx+\left( 1-t\right) y\right) +f\left( \left(
1-t\right) x+ty\right) \right] \leq \max \left\{ f\left( x\right) ,f\left(
y\right) \right\} ,\text{ \ \ \ \ }\left( WQC\right)
\end{equation*}%
or equivalently%
\begin{equation*}
\frac{1}{2}\left[ f\left( y\right) +f\left( x+\delta \right) \right] \leq
\max \left\{ f\left( x\right) ,f\left( y+\delta \right) \right\}
\end{equation*}%
for every $x,$ $y+\delta \in I,$ $x<y$ and $\delta >0.$
\end{definition}

\begin{definition}
(See \cite{SS1}) The mapping $f:I\rightarrow 
%TCIMACRO{\U{211d} }%
%BeginExpansion
\mathbb{R}
%EndExpansion
$ is Jensen- or J-quasi-convex if 
\begin{equation*}
f\left( \frac{x+y}{2}\right) \leq \max \left\{ f(x),f(y)\right\} ,\text{ \ \
\ \ }\left( JQC\right)
\end{equation*}%
for all $x,y\in I.$
\end{definition}

Note that the class $JQC(I)$ of J-quasi-convex functions on $I$ contains the
class $J(I)$ of J-convex functions on $I,$ that is, functions satisfying the
condition%
\begin{equation*}
f\left( \frac{x+y}{2}\right) \leq \frac{f(x)+f(y)}{2},\text{ \ \ \ }\left(
J\right)
\end{equation*}%
for all $x,y\in I.$

In \cite{SS1}, Dragomir and Pearce proved following theorems containing
J-quasi-convex and Wright-quasi-convex functions.

\begin{theorem}
Suppose $a,b\in I\subseteq 
%TCIMACRO{\U{211d} }%
%BeginExpansion
\mathbb{R}
%EndExpansion
$ and $a<b.$ If $f\in JQC\left( I\right) \cap L_{1}\left[ a,b\right] ,$ then 
\begin{equation}
f\left( \frac{a+b}{2}\right) \leq \frac{1}{b-a}\dint%
\nolimits_{a}^{b}f(x)dx+I\left( a,b\right)  \label{1.2}
\end{equation}%
where%
\begin{equation*}
I\left( a,b\right) =\frac{1}{2}\int_{0}^{1}\left\vert f\left( ta+\left(
1-t\right) b\right) -f\left( \left( 1-t\right) a+tb\right) \right\vert dt.
\end{equation*}
\end{theorem}

\begin{theorem}
Let $f:I\rightarrow 
%TCIMACRO{\U{211d} }%
%BeginExpansion
\mathbb{R}
%EndExpansion
$ be a Wright-quasi-convex map on $I$ and suppose $a,b\in I\subseteq 
%TCIMACRO{\U{211d} }%
%BeginExpansion
\mathbb{R}
%EndExpansion
$ with $a<b$ and $f\in L_{1}\left[ a,b\right] ,$ one has the inequality%
\begin{equation}
\frac{1}{b-a}\dint\nolimits_{a}^{b}f(x)dx\leq \max \left\{ f(a),f(b)\right\}
.  \label{1.3}
\end{equation}
\end{theorem}

In \cite{SS1}, Dragomir and Pearce also gave the following theorems
involving some inclusions.

\begin{theorem}
Let $WQC\left( I\right) $ denote the class of Wright-quasi-convex functions
on $I\subseteq 
%TCIMACRO{\U{211d} }%
%BeginExpansion
\mathbb{R}
%EndExpansion
,$ then 
\begin{equation}
QC\left( I\right) \subset WQC\left( I\right) \subset JQC\left( I\right) .
\label{1.4}
\end{equation}%
Both inclusions are proper.
\end{theorem}

\begin{theorem}
We have the inlusions%
\begin{equation}
W\left( I\right) \subset WQC\left( I\right) ,\text{ \ \ }C(I)\subset QC(I),%
\text{ \ \ }J(I)\subset JQC(I).  \label{1.5}
\end{equation}%
Each inclusion is proper.
\end{theorem}

For recent results related to quasi-convex functions see the papers \cite%
{AL1}-\cite{AH} and books \cite{SS2}, \cite{GRE}. In \cite{SS}, Dragomir
defined co-ordinated convex functions and proved following inequalities.

Let us consider the bidimensional interval $\Delta =\left[ a,b\right] \times %
\left[ c,d\right] $ in $%
%TCIMACRO{\U{211d} }%
%BeginExpansion
\mathbb{R}
%EndExpansion
^{2}$ with $a<b$ and $c<d.$ A function $f:\Delta \rightarrow 
%TCIMACRO{\U{211d} }%
%BeginExpansion
\mathbb{R}
%EndExpansion
$ will be called convex on the co-ordinates if the partial mappings%
\begin{equation*}
f_{y}:\left[ a,b\right] \rightarrow 
%TCIMACRO{\U{211d} }%
%BeginExpansion
\mathbb{R}
%EndExpansion
,\text{ \ \ }f_{y}\left( u\right) =f\left( u,y\right)
\end{equation*}%
and%
\begin{equation*}
f_{x}:\left[ c,d\right] \rightarrow 
%TCIMACRO{\U{211d} }%
%BeginExpansion
\mathbb{R}
%EndExpansion
,\text{ \ \ }f_{x}\left( v\right) =f\left( x,v\right)
\end{equation*}%
are convex where defined for all $y\in \left[ c,d\right] $ and $x\in \left[
a,b\right] .$

Recall that the mapping $f:\Delta \rightarrow 
%TCIMACRO{\U{211d} }%
%BeginExpansion
\mathbb{R}
%EndExpansion
$ is convex on $\Delta $, if the following inequality;%
\begin{equation}
f\left( \lambda x+\left( 1-\lambda \right) z,\lambda y+\left( 1-\lambda
\right) w\right) \leq \lambda f\left( x,y\right) +\left( 1-\lambda \right)
f\left( z,w\right)  \label{a}
\end{equation}%
holds for all $\left( x,y\right) ,$ $\left( z,w\right) \in \Delta $ and $%
\lambda \in \left[ 0,1\right] .$

\begin{theorem}
(see \cite{SS}, Theorem 1) Suppose that $f:\Delta =\left[ a,b\right] \times %
\left[ c,d\right] \rightarrow 
%TCIMACRO{\U{211d} }%
%BeginExpansion
\mathbb{R}
%EndExpansion
$ is convex on the co-ordinates on $\Delta .$ Then one has the inequalities;%
\begin{eqnarray}
f\left( \frac{a+b}{2},\frac{c+d}{2}\right) &\leq &\frac{1}{2}\left[ \frac{1}{%
b-a}\dint\nolimits_{a}^{b}f\left( x,\frac{c+d}{2}\right) dx+\frac{1}{d-c}%
\dint\nolimits_{c}^{d}f\left( \frac{a+b}{2},y\right) dy\right]  \notag \\
&\leq &\frac{1}{\left( b-a\right) \left( d-c\right) }\int_{a}^{b}%
\int_{c}^{d}f\left( x,y\right) dydx  \label{1.6} \\
&\leq &\frac{1}{4}\left[ \frac{1}{b-a}\dint\nolimits_{a}^{b}f\left(
x,c\right) dx+\frac{1}{b-a}\dint\nolimits_{a}^{b}f\left( x,d\right) dx\right.
\notag \\
&&\left. \frac{1}{d-c}\dint\nolimits_{c}^{d}f\left( a,y\right) dy+\frac{1}{%
d-c}\dint\nolimits_{c}^{d}f\left( b,y\right) dy\right]  \notag \\
&\leq &\frac{f\left( a,c\right) +f\left( b,c\right) +f\left( a,d\right)
+f\left( b,d\right) }{4}  \notag
\end{eqnarray}%
The above inequalities are sharp.
\end{theorem}

Similar results can be found in \cite{AL}-\cite{BAK}.

This paper is arranged as follows. Firstly, we will give some definitions on
quasi-convex functions and lemmas belong to this definitions. Secondly, we
will prove several inequalities contain co-ordinated quasi-convex functions.
Also, we will discuss the inclusions a connection with some different
classes of co-ordinated convex functions.

\section{DEFINITIONS AND MAIN\ RESULTS}

We will start the following definitions and lemmas;

\begin{definition}
A function $f:\Delta =\left[ a,b\right] \times \left[ c,d\right] \rightarrow 
%TCIMACRO{\U{211d} }%
%BeginExpansion
\mathbb{R}
%EndExpansion
$ is said quasi-convex function on the co-ordinates on $\Delta $ if the
following inequality%
\begin{equation*}
f\left( \lambda x+\left( 1-\lambda \right) z,\lambda y+\left( 1-\lambda
\right) w\right) \leq \max \left\{ f\left( x,y\right) ,f\left( z,w\right)
\right\}
\end{equation*}%
holds for all $\left( x,y\right) ,$ $\left( z,w\right) \in \Delta $ and $%
\lambda \in \left[ 0,1\right] $
\end{definition}

$f:\Delta \rightarrow 
%TCIMACRO{\U{211d} }%
%BeginExpansion
\mathbb{R}
%EndExpansion
$ will be called co-ordinated quasi-convex on the co-ordinates if the
partial mappings%
\begin{equation*}
f_{y}:\left[ a,b\right] \rightarrow 
%TCIMACRO{\U{211d} }%
%BeginExpansion
\mathbb{R}
%EndExpansion
,\text{ \ \ }f_{y}\left( u\right) =f\left( u,y\right)
\end{equation*}%
and%
\begin{equation*}
f_{x}:\left[ c,d\right] \rightarrow 
%TCIMACRO{\U{211d} }%
%BeginExpansion
\mathbb{R}
%EndExpansion
,\text{ \ \ }f_{x}\left( v\right) =f\left( x,v\right)
\end{equation*}%
are convex where defined for all $y\in \left[ c,d\right] $ and $x\in \left[
a,b\right] .$ We denote by $QC(\Delta )$ the classes of quasi-convex
functions on the co-ordinates on $\Delta .$ The following lemma holds.

\begin{lemma}
Every quasi-convex mapping $f:\Delta \rightarrow 
%TCIMACRO{\U{211d} }%
%BeginExpansion
\mathbb{R}
%EndExpansion
$ is quasi-convex on the co-ordinates.
\end{lemma}

\begin{proof}
Suppose that $f:\Delta =\left[ a,b\right] \times \left[ c,d\right]
\rightarrow 
%TCIMACRO{\U{211d} }%
%BeginExpansion
\mathbb{R}
%EndExpansion
$ is quasi-convex on $\Delta .$ Then the partial mappings 
\begin{equation*}
f_{y}:\left[ a,b\right] \rightarrow 
%TCIMACRO{\U{211d} }%
%BeginExpansion
\mathbb{R}
%EndExpansion
,\text{ \ \ }f_{y}\left( u\right) =f\left( u,y\right) ,\text{ \ \ }y\in %
\left[ c,d\right]
\end{equation*}%
and%
\begin{equation*}
f_{x}:\left[ c,d\right] \rightarrow 
%TCIMACRO{\U{211d} }%
%BeginExpansion
\mathbb{R}
%EndExpansion
,\text{ \ \ }f_{x}\left( v\right) =f\left( x,v\right) ,\text{ \ \ }x\in %
\left[ a,b\right] \text{\ }
\end{equation*}%
are convex on $\Delta .$ For $\lambda \in \left[ 0,1\right] $ and $%
v_{1},v_{2}\in \left[ c,d\right] ,$ one has%
\begin{eqnarray*}
f_{x}\left( \lambda v_{1}+\left( 1-\lambda \right) v_{2}\right) &=&f\left(
x,\lambda v_{1}+\left( 1-\lambda \right) v_{2}\right) \\
&=&f\left( \lambda x+\left( 1-\lambda \right) x,\lambda v_{1}+\left(
1-\lambda \right) v_{2}\right) \\
&\leq &\max \left\{ f\left( x,v_{1}\right) ,f\left( x,v_{2}\right) \right\}
\\
&=&\max \left\{ f_{x}\left( v_{1}\right) ,f_{x}\left( v_{2}\right) \right\}
\end{eqnarray*}%
which completes the proof of quasi-convexity of $f_{x}$ on $\left[ c,d\right]
.$ Therefore $f_{y}:\left[ a,b\right] \rightarrow 
%TCIMACRO{\U{211d} }%
%BeginExpansion
\mathbb{R}
%EndExpansion
,$ \ \ $f_{y}\left( u\right) =f\left( u,y\right) $ is also quasi-convex on $%
\left[ a,b\right] $ for all $y\in \left[ c,d\right] ,$ goes likewise and we
shall omit the details.
\end{proof}

\begin{definition}
A function $f:\Delta =\left[ a,b\right] \times \left[ c,d\right] \rightarrow 
%TCIMACRO{\U{211d} }%
%BeginExpansion
\mathbb{R}
%EndExpansion
$ is said J-convex function on the co-ordinates on $\Delta $ if the
following inequality%
\begin{equation*}
f\left( \frac{x+z}{2},\frac{y+w}{2}\right) \leq \frac{f\left( x,y\right)
+f\left( z,w\right) }{2}
\end{equation*}%
holds for all $\left( x,y\right) ,$ $\left( z,w\right) \in \Delta .$ We
denote by $J(\Delta )$ the classes of J-convex functions on the co-ordinates
on $\Delta $
\end{definition}

\begin{lemma}
Every J-convex mapping defined $f:\Delta \rightarrow 
%TCIMACRO{\U{211d} }%
%BeginExpansion
\mathbb{R}
%EndExpansion
$ is J-convex on the co-ordinates.
\end{lemma}

\begin{proof}
By the partial mappings, we can write for $v_{1},v_{2}\in \left[ c,d\right]
, $%
\begin{eqnarray*}
f_{x}\left( \frac{v_{1}+v_{2}}{2}\right) &=&f\left( x,\frac{v_{1}+v_{2}}{2}%
\right) \\
&=&f\left( \frac{x+x}{2},\frac{v_{1}+v_{2}}{2}\right) \\
&\leq &\frac{f\left( x,v_{1}\right) +f\left( x,v_{2}\right) }{2} \\
&=&\frac{f_{x}\left( v_{1}\right) +f_{x}\left( v_{2}\right) }{2}
\end{eqnarray*}%
which completes the proof of J-convexity of $f_{x}$ on $\left[ c,d\right] .$
Similarly, we can prove J-convexity of $f_{y}$ on $\left[ a,b\right] .$
\end{proof}

\begin{definition}
A function $f:\Delta =\left[ a,b\right] \times \left[ c,d\right] \rightarrow 
%TCIMACRO{\U{211d} }%
%BeginExpansion
\mathbb{R}
%EndExpansion
$ is said J-quasi-convex function on the co-ordinates on $\Delta $ if the
following inequality%
\begin{equation*}
f\left( \frac{x+z}{2},\frac{y+w}{2}\right) \leq \max \left\{ f\left(
x,y\right) ,f\left( z,w\right) \right\}
\end{equation*}%
holds for all $\left( x,y\right) ,$ $\left( z,w\right) \in \Delta .$ We
denote by $JQC(\Delta )$ the classes of J-quasi-convex functions on the
co-ordinates on $\Delta $
\end{definition}

\begin{lemma}
Every J-quasi-convex mapping defined $f:\Delta \rightarrow 
%TCIMACRO{\U{211d} }%
%BeginExpansion
\mathbb{R}
%EndExpansion
$ is J-quasi-convex on the co-ordinates.
\end{lemma}

\begin{proof}
By a similar way to proof of Lemma 1, we can write for $v_{1},v_{2}\in \left[
c,d\right] ,$%
\begin{eqnarray*}
f_{x}\left( \frac{v_{1}+v_{2}}{2}\right) &=&f\left( x,\frac{v_{1}+v_{2}}{2}%
\right) \\
&=&f\left( \frac{x+x}{2},\frac{v_{1}+v_{2}}{2}\right) \\
&\leq &\max \left\{ f\left( x,v_{1}\right) ,f\left( x,v_{2}\right) \right\}
\\
&=&\max \left\{ f_{x}\left( v_{1}\right) ,f_{x}\left( v_{2}\right) \right\}
\end{eqnarray*}%
which completes the proof of J-quasi-convexity of $f_{x}$ on $\left[ c,d%
\right] .$ We can also prove J-quasi-convexity of $f_{y}$ on $\left[ a,b%
\right] .$
\end{proof}

\begin{definition}
A function $f:\Delta =\left[ a,b\right] \times \left[ c,d\right] \rightarrow 
%TCIMACRO{\U{211d} }%
%BeginExpansion
\mathbb{R}
%EndExpansion
$ is said Wright-convex function on the co-ordinates on $\Delta $ if the
following inequality%
\begin{equation*}
f\left( \left( 1-t\right) a+tb,\left( 1-s\right) c+sd\right) +f\left(
ta+\left( 1-t\right) b,sc+\left( 1-s\right) d\right) \leq f\left( a,c\right)
+f\left( b,d\right)
\end{equation*}%
holds for all $\left( a,c\right) ,$ $\left( b,d\right) \in \Delta $ and $%
t,s\in \left[ 0,1\right] .$ We denote by $W(\Delta )$ the classes of
Wright-convex functions on the co-ordinates on $\Delta $
\end{definition}

\begin{lemma}
Every Wright-convex mapping defined $f:\Delta \rightarrow 
%TCIMACRO{\U{211d} }%
%BeginExpansion
\mathbb{R}
%EndExpansion
$ is Wright-convex on the co-ordinates.
\end{lemma}

\begin{proof}
Suppose that $f:\Delta \rightarrow 
%TCIMACRO{\U{211d} }%
%BeginExpansion
\mathbb{R}
%EndExpansion
$ is Wright-convex on $\Delta $. Then by partial mapping, for $%
v_{1},v_{2}\in \left[ c,d\right] ,$ $x\in \left[ a,b\right] ,$%
\begin{eqnarray*}
&&f_{x}\left( \left( 1-t\right) v_{1}+tv_{2}\right) +f_{x}\left(
tv_{1}+\left( 1-t\right) v_{2}\right) \\
&=&f\left( x,\left( 1-t\right) v_{1}+tv_{2}\right) +f\left( x,tv_{1}+\left(
1-t\right) v_{2}\right) \\
&=&f\left( \left( 1-t\right) x+tx,\left( 1-t\right) v_{1}+tv_{2}\right)
+f\left( tx+\left( 1-t\right) x,tv_{1}+\left( 1-t\right) v_{2}\right) \\
&\leq &f\left( x,v_{1}\right) +f\left( x,v_{2}\right) \\
&=&f_{x}\left( v_{1}\right) +f_{x}\left( v_{2}\right)
\end{eqnarray*}%
which shows that $f_{x}$ is Wright-convex on $\left[ c,d\right] .$ Similarly
one can see that $f_{y}$ is Wright-convex on $\left[ a,b\right] .$
\end{proof}

\begin{definition}
A function $f:\Delta =\left[ a,b\right] \times \left[ c,d\right] \rightarrow 
%TCIMACRO{\U{211d} }%
%BeginExpansion
\mathbb{R}
%EndExpansion
$ is said Wright-quasi-convex function on the co-ordinates on $\Delta $ if
the following inequality%
\begin{equation*}
\frac{1}{2}\left[ f\left( tx+\left( 1-t\right) z,ty+\left( 1-t\right)
w\right) +f\left( \left( 1-t\right) x+tz,\left( 1-t\right) y+tw\right) %
\right] \leq \max \left\{ f\left( x,y\right) ,f\left( z,w\right) \right\}
\end{equation*}%
holds for all $\left( x,y\right) ,$ $\left( z,w\right) \in \Delta $ and $%
t\in \left[ 0,1\right] .$ We denote by $WQC(\Delta )$ the classes of
Wright-quasi-convex functions on the co-ordinates on $\Delta $
\end{definition}

\begin{lemma}
Every Wright-quasi-convex mapping defined $f:\Delta \rightarrow 
%TCIMACRO{\U{211d} }%
%BeginExpansion
\mathbb{R}
%EndExpansion
$ is Wright-quasi-convex on the co-ordinates.
\end{lemma}

\begin{proof}
Suppose that $f:\Delta \rightarrow 
%TCIMACRO{\U{211d} }%
%BeginExpansion
\mathbb{R}
%EndExpansion
$ is Wright-quasi-convex on $\Delta $. Then by partial mapping, for $%
v_{1},v_{2}\in \left[ c,d\right] ,$%
\begin{eqnarray*}
&&\frac{1}{2}\left[ f_{x}\left( tv_{1}+\left( 1-t\right) v_{2}\right)
+f_{x}\left( \left( 1-t\right) v_{1}+tv_{2}\right) \right] \\
&=&\frac{1}{2}\left[ f\left( x,tv_{1}+\left( 1-t\right) v_{2}\right)
+f\left( x,\left( 1-t\right) v_{1}+tv_{2}\right) \right] \\
&=&\frac{1}{2}\left[ f\left( tx+\left( 1-t\right) x,tv_{1}+\left( 1-t\right)
v_{2}\right) +f\left( \left( 1-t\right) x+tx,\left( 1-t\right)
v_{1}+tv_{2}\right) \right] \\
&\leq &\max \left\{ f\left( x,v_{1}\right) ,f\left( x,v_{2}\right) \right\}
\\
&=&\max \left\{ f_{x}\left( v_{1}\right) ,f_{x}\left( v_{2}\right) \right\}
\end{eqnarray*}%
which shows that $f_{x}$ is Wright-quasi-convex on $\left[ c,d\right] .$
Similarly one can see that $f_{y}$ is Wright-quasi-convex on $\left[ a,b%
\right] .$
\end{proof}

\begin{theorem}
Suppose that $f:\Delta =\left[ a,b\right] \times \left[ c,d\right]
\rightarrow 
%TCIMACRO{\U{211d} }%
%BeginExpansion
\mathbb{R}
%EndExpansion
$ is J-quasi-convex on the co-ordinates on $\Delta .$ If $f_{x}\in L_{1}%
\left[ c,d\right] $ and $f_{y}\in L_{1}\left[ a,b\right] ,$ then we have the
inequality;%
\begin{eqnarray}
&&\frac{1}{2}\left[ \frac{1}{b-a}\dint\nolimits_{a}^{b}f\left( x,\frac{c+d}{2%
}\right) dx+\frac{1}{d-c}\int_{c}^{d}f\left( \frac{a+b}{2},y\right) dy\right]
\label{2.1} \\
&\leq &\frac{1}{\left( b-a\right) \left( d-c\right) }\int_{c}^{d}\dint%
\nolimits_{a}^{b}f(x,y)dxdy+H(x,y)  \notag
\end{eqnarray}%
where%
\begin{eqnarray*}
H\left( x,y\right) &=&\frac{1}{4\left( d-c\right) }\int_{c}^{d}\int_{0}^{1}%
\left\vert f\left( ta+\left( 1-t\right) b,y\right) -f\left( \left(
1-t\right) a+tb,y\right) \right\vert dtdy \\
&&+\frac{1}{4\left( b-a\right) }\dint\nolimits_{a}^{b}\int_{0}^{1}\left\vert
f\left( x,tc+\left( 1-t\right) d\right) -f\left( x,\left( 1-t\right)
c+td\right) \right\vert dtdx.
\end{eqnarray*}
\end{theorem}

\begin{proof}
Since $f:\Delta \rightarrow 
%TCIMACRO{\U{211d} }%
%BeginExpansion
\mathbb{R}
%EndExpansion
$ is J-quasi-convex on the co-ordinates on $\Delta .$ We can write the
partial mappings%
\begin{equation*}
f_{y}:\left[ a,b\right] \rightarrow 
%TCIMACRO{\U{211d} }%
%BeginExpansion
\mathbb{R}
%EndExpansion
,\text{ \ \ }f_{y}\left( u\right) =f\left( u,y\right) ,\text{ \ \ }y\in %
\left[ c,d\right]
\end{equation*}%
and%
\begin{equation*}
f_{x}:\left[ c,d\right] \rightarrow 
%TCIMACRO{\U{211d} }%
%BeginExpansion
\mathbb{R}
%EndExpansion
,\text{ \ \ }f_{x}\left( v\right) =f\left( x,v\right) ,\text{ \ \ }x\in %
\left[ a,b\right]
\end{equation*}%
are J-quasi-convex on $\Delta .$ Then by the inequality (\ref{1.2}), we have%
\begin{equation*}
f_{y}\left( \frac{a+b}{2}\right) \leq \frac{1}{b-a}\dint%
\nolimits_{a}^{b}f_{y}(x)dx+\frac{1}{2}\int_{0}^{1}\left\vert f_{y}\left(
ta+\left( 1-t\right) b\right) -f_{y}\left( \left( 1-t\right) a+tb\right)
\right\vert dt.
\end{equation*}%
That is%
\begin{equation*}
f\left( \frac{a+b}{2},y\right) \leq \frac{1}{b-a}\dint%
\nolimits_{a}^{b}f(x,y)dx+\frac{1}{2}\int_{0}^{1}\left\vert f\left(
ta+\left( 1-t\right) b,y\right) -f\left( \left( 1-t\right) a+tb,y\right)
\right\vert dt.
\end{equation*}%
Integrating the resulting inequality with respect to $y$ over $\left[ c,d%
\right] $ and dividing both sides of inequality with $\left( d-c\right) ,$
we get%
\begin{eqnarray}
&&\frac{1}{d-c}\int_{c}^{d}f\left( \frac{a+b}{2},y\right) dy  \label{2.2} \\
&\leq &\frac{1}{\left( b-a\right) \left( d-c\right) }\int_{c}^{d}\dint%
\nolimits_{a}^{b}f(x,y)dxdy  \notag \\
&&+\frac{1}{2\left( d-c\right) }\int_{c}^{d}\int_{0}^{1}\left\vert f\left(
ta+\left( 1-t\right) b,y\right) -f\left( \left( 1-t\right) a+tb,y\right)
\right\vert dtdy.  \notag
\end{eqnarray}%
By a similar argument, we have%
\begin{eqnarray}
&&\frac{1}{b-a}\dint\nolimits_{a}^{b}f\left( x,\frac{c+d}{2}\right) dx
\label{2.3} \\
&\leq &\frac{1}{\left( b-a\right) \left( d-c\right) }\dint\nolimits_{a}^{b}%
\dint\nolimits_{c}^{d}f(x,y)dydx  \notag \\
&&+\frac{1}{2\left( b-a\right) }\dint\nolimits_{a}^{b}\int_{0}^{1}\left\vert
f\left( x,tc+\left( 1-t\right) d\right) -f\left( x,\left( 1-t\right)
c+td\right) \right\vert dtdx.  \notag
\end{eqnarray}%
Summing (\ref{2.2}) and (\ref{2.3}), we get the required result.
\end{proof}

\begin{theorem}
Suppose that $f:\Delta =\left[ a,b\right] \times \left[ c,d\right]
\rightarrow 
%TCIMACRO{\U{211d} }%
%BeginExpansion
\mathbb{R}
%EndExpansion
$ is Wright-quasi-convex on the co-ordinates on $\Delta .$ If $f_{x}\in L_{1}%
\left[ c,d\right] $ and $f_{y}\in L_{1}\left[ a,b\right] ,$ then we have the
inequality;%
\begin{eqnarray}
&&\frac{1}{\left( b-a\right) \left( d-c\right) }\dint\nolimits_{c}^{d}\dint%
\nolimits_{a}^{b}f(x,y)dxdy  \label{2.4} \\
&\leq &\frac{1}{2}\left[ \max \left\{ \frac{1}{\left( b-a\right) }%
\dint\nolimits_{a}^{b}f(x,c)dx,\frac{1}{\left( b-a\right) }%
\dint\nolimits_{a}^{b}f(x,d)dx\right\} \right.  \notag \\
&&\left. +\max \left\{ \frac{1}{\left( d-c\right) }\dint%
\nolimits_{c}^{d}f(a,y)dy,\frac{1}{\left( d-c\right) }\dint%
\nolimits_{c}^{d}f(b,y)dy\right\} \right] .  \notag
\end{eqnarray}
\end{theorem}

\begin{proof}
Since $f:\Delta \rightarrow 
%TCIMACRO{\U{211d} }%
%BeginExpansion
\mathbb{R}
%EndExpansion
$ is Wright-quasi-convex on the co-ordinates on $\Delta .$ We can write the
partial mappings%
\begin{equation*}
f_{y}:\left[ a,b\right] \rightarrow 
%TCIMACRO{\U{211d} }%
%BeginExpansion
\mathbb{R}
%EndExpansion
,\text{ \ \ }f_{y}\left( u\right) =f\left( u,y\right) ,\text{ \ \ }y\in %
\left[ c,d\right]
\end{equation*}%
and%
\begin{equation*}
f_{x}:\left[ c,d\right] \rightarrow 
%TCIMACRO{\U{211d} }%
%BeginExpansion
\mathbb{R}
%EndExpansion
,\text{ \ \ }f_{x}\left( v\right) =f\left( x,v\right) ,\text{ \ \ }x\in %
\left[ a,b\right]
\end{equation*}%
are Wright-quasi-convex on $\Delta .$ Then by the inequality (\ref{1.3}), we
have%
\begin{equation*}
\frac{1}{b-a}\dint\nolimits_{a}^{b}f_{y}(x)dx\leq \max \left\{
f_{y}(a),f_{y}(b)\right\} .
\end{equation*}%
That is%
\begin{equation*}
\frac{1}{b-a}\dint\nolimits_{a}^{b}f(x,y)dx\leq \max \left\{
f(a,y),f(b,y)\right\} .
\end{equation*}%
Dividing both sides of inequality with $\left( d-c\right) $ and integrating
with respect to $y$ over $\left[ c,d\right] $ $,$ we get%
\begin{equation}
\frac{1}{\left( b-a\right) \left( d-c\right) }\dint\nolimits_{c}^{d}\dint%
\nolimits_{a}^{b}f(x,y)dxdy\leq \max \left\{ \frac{1}{\left( d-c\right) }%
\dint\nolimits_{c}^{d}f(a,y)dy,\frac{1}{\left( d-c\right) }%
\dint\nolimits_{c}^{d}f(b,y)dy\right\} .  \label{2.5}
\end{equation}%
By a similar argument, we can write%
\begin{equation}
\frac{1}{\left( b-a\right) \left( d-c\right) }\dint\nolimits_{c}^{d}\dint%
\nolimits_{a}^{b}f(x,y)dxdy\leq \max \left\{ \frac{1}{\left( b-a\right) }%
\dint\nolimits_{a}^{b}f(x,c)dx,\frac{1}{\left( b-a\right) }%
\dint\nolimits_{a}^{b}f(x,d)dx\right\} .  \label{2.6}
\end{equation}%
By addition (\ref{2.5}) and (\ref{2.6}), we have%
\begin{eqnarray*}
&&\frac{1}{\left( b-a\right) \left( d-c\right) }\dint\nolimits_{c}^{d}\dint%
\nolimits_{a}^{b}f(x,y)dxdy \\
&\leq &\frac{1}{2}\left[ \max \left\{ \frac{1}{\left( b-a\right) }%
\dint\nolimits_{a}^{b}f(x,c)dx,\frac{1}{\left( b-a\right) }%
\dint\nolimits_{a}^{b}f(x,d)dx\right\} \right. \\
&&\left. +\max \left\{ \frac{1}{\left( d-c\right) }\dint%
\nolimits_{c}^{d}f(a,y)dy,\frac{1}{\left( d-c\right) }\dint%
\nolimits_{c}^{d}f(b,y)dy\right\} \right]
\end{eqnarray*}%
which completes the proof.
\end{proof}

\begin{theorem}
Let $C\left( \Delta \right) ,$ $J\left( \Delta \right) ,$ $W\left( \Delta
\right) ,$ $QC\left( \Delta \right) ,$ $JQC\left( \Delta \right) ,$ $%
WQC\left( \Delta \right) $ denote the classes of functions co-ordinated
convex, co-ordinated J-convex, co-ordinated W-convex, co-ordinated
quasi-convex, co-ordinated J-quasi-convex and co-ordinated W-quasi-convex
functions on $\Delta =\left[ a,b\right] \times \left[ c,d\right] $,
respectively, we have following inclusions.%
\begin{equation}
QC\left( \Delta \right) \subset WQC\left( \Delta \right) \subset JQC\left(
\Delta \right)  \label{2.7}
\end{equation}%
\begin{equation}
W\left( \Delta \right) \subset WQC\left( \Delta \right) ,\text{ \ \ }C\left(
\Delta \right) \subset J\left( \Delta \right) ,\text{ \ \ }J\left( \Delta
\right) \subset JQC\left( \Delta \right) .  \label{2.8}
\end{equation}
\end{theorem}

\begin{proof}
Let $f\in QC\left( \Delta \right) .$ Then for all $\left( x,y\right) ,$ $%
\left( z,w\right) \in \Delta $ and $t\in \left[ 0,1\right] ,$ we have%
\begin{equation*}
f\left( \lambda x+\left( 1-\lambda \right) z,\lambda y+\left( 1-\lambda
\right) w\right) \leq \max \left\{ f\left( x,y\right) ,f\left( z,w\right)
\right\}
\end{equation*}%
\begin{equation*}
f\left( \left( 1-\lambda \right) x+\lambda z,\left( 1-\lambda \right)
y+\lambda w\right) \leq \max \left\{ f\left( x,y\right) ,f\left( z,w\right)
\right\} .
\end{equation*}%
By addition, we obtain%
\begin{eqnarray}
&&\frac{1}{2}\left[ f\left( \lambda x+\left( 1-\lambda \right) z,\lambda
y+\left( 1-\lambda \right) w\right) +f\left( \left( 1-\lambda \right)
x+\lambda z,\left( 1-\lambda \right) y+\lambda w\right) \right]  \label{2.9}
\\
&\leq &\max \left\{ f\left( x,y\right) ,f\left( z,w\right) \right\}  \notag
\end{eqnarray}%
that is, $f\in WQC\left( \Delta \right) .$ In (\ref{2.9}), if we choose $%
\lambda =\frac{1}{2},$ we obtain $WQC\left( \Delta \right) \subset JQC\left(
\Delta \right) .$ Which completes the proof of (\ref{2.7}).

In order to prove (\ref{2.8}), taking $f\in W\left( \Delta \right) $ and
using the definition, we get%
\begin{equation*}
\frac{1}{2}\left[ f\left( \left( 1-t\right) a+tb,\left( 1-s\right)
c+sd\right) +f\left( ta+\left( 1-t\right) b,sc+\left( 1-s\right) d\right) %
\right] \leq \frac{f\left( a,c\right) +f\left( b,d\right) }{2}
\end{equation*}%
for all $\left( a,c\right) ,\left( b,d\right) \in \Delta $ and $t\in \left[
0,1\right] .$ Using the fact that%
\begin{equation*}
\frac{f\left( a,c\right) +f\left( b,d\right) +\left\vert f\left( a,c\right)
-f\left( b,d\right) \right\vert }{2}=\max \left\{ f(a,c),f(b,d)\right\}
\end{equation*}%
we can write 
\begin{equation*}
\frac{f\left( a,c\right) +f\left( b,d\right) }{2}\leq \max \left\{
f(a,c),f(b,d)\right\}
\end{equation*}%
for all $\left( a,c\right) ,\left( b,d\right) \in \Delta ,$ we obtain $%
W\left( \Delta \right) \subset WQC\left( \Delta \right) .$

Taking $f\in C\left( \Delta \right) $ and, if we choose $t=\frac{1}{2}$ in (%
\ref{a})$,$ we obtain 
\begin{equation*}
f\left( \frac{x+z}{2},\frac{y+w}{2}\right) \leq \frac{f\left( x,y\right)
+f\left( z,w\right) }{2}
\end{equation*}%
for all $\left( x,y\right) ,$ $\left( z,w\right) \in \Delta .$ One can see
that $C\left( \Delta \right) \subset J\left( \Delta \right) .$

Taking $f\in J\left( \Delta \right) ,$ we can write%
\begin{equation*}
f\left( \frac{x+z}{2},\frac{y+w}{2}\right) \leq \frac{f\left( x,y\right)
+f\left( z,w\right) }{2}
\end{equation*}%
for all $\left( x,y\right) ,$ $\left( z,w\right) \in \Delta .$ Using the
fact that%
\begin{equation*}
\frac{f\left( x,y\right) +f\left( z,w\right) +\left\vert f\left( x,y\right)
-f\left( z,w\right) \right\vert }{2}=\max \left\{ f(x,y),f(z,w)\right\}
\end{equation*}%
we can write 
\begin{equation*}
\frac{f\left( x,y\right) +f\left( z,w\right) }{2}\leq \max \left\{ f\left(
x,y\right) ,f\left( z,w\right) \right\} .
\end{equation*}%
Then obviously, we obtain%
\begin{equation*}
f\left( \frac{x+z}{2},\frac{y+w}{2}\right) \leq \max \left\{ f\left(
x,y\right) ,f\left( z,w\right) \right\}
\end{equation*}%
which shows that $f\in JQ\left( \Delta \right) .$
\end{proof}

\end{document}